\documentclass{article}
\usepackage{lmodern}
\usepackage[utf8]{inputenc}
\usepackage{latexcolors}  
\usepackage{amsmath, amsthm, amssymb, hyperref}
\usepackage[margin=1in]{geometry}
\usepackage{enumerate}
\usepackage[shortlabels]{enumitem}
\usepackage{mathrsfs}
\usepackage{mathtools}
\usepackage{multicol}
\usepackage{graphicx, float} 
\usepackage[capitalise, noabbrev]{cleveref}
\usepackage{venndiagram}
\usepackage{tikz-cd}
\usepackage[normalem]{ulem}
\usepackage{verbatim}
\usepackage{caption}
\usepackage{subcaption}

\usepackage{todonotes}

\newcommand{\qif}{\quad \mbox{if} \quad}

\DeclareMathOperator{\conv}{conv}

\newtheorem{theorem}{Theorem}[section]

\newtheorem{corollary}[theorem]{Corollary}
\newtheorem{lemma}[theorem]{Lemma}
\theoremstyle{definition}
\newtheorem{definition}[theorem]{Definition}
\newtheorem{example}[theorem]{Example}
\newtheorem{remark}[theorem]{Remark}

\newcommand{\Cos}{\mathcal{C}}
\newcommand{\DuCos}{\mathcal{C}^\circ}
\newcommand{\x}{\mathbf{x}}
\newcommand{\y}{\mathbf{y}}
\newcommand{\R}{\mathbb{R}}
\newcommand{\p}{\mathbf{p}}

\title{On cosmological polytopes, their canonical forms and their duals}
\author{Anna Birkemeyer \thanks{Universit\"at Bonn, Germany. \texttt{anna.birkemeyer@gmx.de}}
\and
Torben Donzelmann \thanks{Universit\"at Osnabr\"uck, Germany. \texttt{torben.donzelmann@uos.de}}
\and Mieke Fink \thanks{Universit\"at Osnabr\"uck, Germany. \texttt{mieke.fink@uos.de}} 
\and 
Martina Juhnke \thanks{Universit\"at Osnabr\"uck, Germany. \texttt{martina.juhnke@uos.de}}}

\begin{document}
\maketitle
\begin{abstract}
    We compute the canonical form of the cosmological polytope for any graph in terms of the dual of the shifted cosmological polytope in two different ways. On the way, we provide  an explicit coordinate description of the dual of the cosmological polytope. Moreover, we construct two  triangulations of the dual cosmological polytope in terms of maximal and almost maximal tubings of the underlying graph. Though the existence of the first triangulation was already suggested by Arkani-Hamed, Benincasa and Postnikov, the second is completely new and, in particular, gives rise to a new expression of the canonical form of the cosmological polytope.
\end{abstract}
\section{Introduction}

The cosmological polytope is a combinatorial object of interest in particle physics. It was first defined by Arkani-Hamed, Benincasa, and Postnikov \cite{arkani17} in their study of the wave function of the universe. The integrands appearing in this wave function correspond to Feynman diagrams, which in our context are simply graphs. To each such graph $G$, they associate a polytope $\Cos_G$, the cosmological polytope, whose \emph{canonical form} coincides exactly with the corresponding integrand. This differential form is a fundamental object in the theory of positive geometries, and there are various methods known to compute it. Replacing integrands by canonical forms hence allows one to apply tools from positive geometry to the study of the wave function. A common approach to compute the canonical form of a polytope is  via triangulations of the polytopes itself, and this has been pursued also for cosmological polytopes, see \cite{JSV} and also \cite{benjes2025}.  We also refer the reader to a recent preprint by Frost and Lotter, where a generalization of the wave function of the universe is studied \cite{FL}.

In this paper, we take a different route via the volume of the dual of the shifted polytope. More precisely, the following statement  serves as our starting point (\cite[Theorem 5.3]{gao2024}, see also \cite[Section 7.4.1]{ArkaniLam17} and \cite[Theorem 2.5]{Telen}): 

\begin{theorem}\label{canonical}\cite{gao2024}
Let $P$ be a $d$-dimensional polytope in $\mathbb{R}^d \subseteq \mathbb{P}^d$. Then the canonical form of $P$ is given by
\[
\Omega(P)(x) = \operatorname{vol}\bigl((P-x)^\circ\bigr)\,\mathrm{d}(x),
\label{expression_can}
\]
where $(P-x)^\circ$ denotes the dual of the shifted polytope $P-x$.
\end{theorem}

We will use this result to study the canonical form of the cosmological polytope in terms of the volume of its dual polytope. The terms appearing in our computations will correspond naturally to combinatorial structures associated with the graph $G$. 
A technical difficulty arises from the fact that the cosmological polytope of a graph $G=(V,E)$ naturally lives in a real vector space of dimension $|E|+|V|$, but has codimension one. Consequently, it must be regarded as a polytope in an appropriate codimension-one subspace in order to define its polar dual. This requires careful treatment to ensure that the construction is geometrically valid—so that Theorem~\ref{canonical} applies—while simultaneously still referencing the combinatorial structure of the graph in its coordinates.

We provide such a description of the dual polytope, which makes this structure explicit and prove that our definition indeed yields the polar dual of the cosmological polytope.

As vertices of the polar dual of a polytope correspond to facets of the original polytope, knowing the facet description of the polytope $\Cos_G$ is crucial for us. 
Indeed, it was shown in \cite{arkani17}, that facets naturally correspond to connected subgraphs of the underlying graph.  

More concretely, for any (not necessarily induced) connected subgraph $t=(V(t),E(t))$ of $G=(V,E)$, the vector
\[
h_t = \sum_{v\in V(t)} \mathbf{x}_v
      + \sum_{e\in E\setminus E(t)} |e\cap V(t)| \cdot \mathbf{y}_e
      \;\in\; \mathbb{P}^{|E|+|V|-1},
\]
is a facet normal of $\Cos_G$. Here, we we denote the standard basis vectors by $\mathbf{x}_v$ for $v\in V$ and $\mathbf{y}_e$ for $e\in E$. 

To compute the volume of a polytope, a natural approach is to determine a triangulation. In \cite{arkani17},~--~without proof~--~ the authors suggested that the dual of the cosmological polytope admits a triangulation in terms of \emph{tubings} of a graph, where, in our nomenclature, a tubing is a collection of connected subgraphs that is totally ordered by inclusion. More precisely, they claimed the following statement: 

\begin{theorem}\label{t:tubingTrig}
Let $G$ be a graph with maximal tubings $T_1,\dots,T_s$, and let
\[
\sigma_i = \conv\left(\frac{h_t}{|h_t|_1}~ |~ t \in T_i \right).
\]
Then $\{ \sigma_1,\dots,\sigma_s \}$ is a triangulation of $\DuCos_G$.
\end{theorem}
We provide a mathematical proof of this result. 
Moreover, by restricting the triangulation in \Cref{t:tubingTrig} and coning over any interior point, we obtain another  triangulation of $\DuCos_G$ in terms of certain \emph{almost maximal tubings}. 

Both triangulations allow us to compute the canonical form of the cosmological polytope, since the combinatorial type of the triangulation of the dual of a polytope is invariant under translation. 

We want to remark that, though the seminal paper by Arkani-Hamed, Benincasa and Postnikov \cite{arkani17} has triggered a lot of research on the cosmological polytope itself (see \cite{KM2025} for results on their face structure, and  \cite{JSV,bruckamp2025,benjes2025} for the study of Ehrhart theoretic properties), their duals have not been considered so far.

The structure of the paper is as follows.
In the next section, we introduce the cosmological polytope and recall known results about its geometry. We also define tubings of a graph in our setting, as this notion plays a central role in describing the triangulation constructed later. We then explain how to compute the shifted dual volume appearing in \Cref{canonical} using a triangulation of $P^\circ$.

In \Cref{sect:triang}, we provide a convenient definition of the dual cosmological polytope $\DuCos_G$ and prove \Cref{t:tubingTrig}. Finally, in  \Cref{section_comp}, we use the results of \Cref{sect:triang} to compute the canonical form of the cosmological polytope in  \Cref{canonicalform}. We show that the induced triangulation on the boundary leads to yet another triangulation of the dual of the cosmological polytope and, hence, yet another representation of the canonical form of $\Cos_G$.

Shortly before publishing this article, we were informed that \Cref{t:tubingTrig} was proven independently and using different methods in \cite{Tyler}. 

\section{Preliminaries}
In this section, we provide the necessary background on graph tubings, cosmological polytopes, as well as on canonical forms and volumes of duals of polytopes.

\subsection{Tubings}

In this subsection, we define the notion of a tubing. This will be helpful to construct the triangulation of the polytope $\DuCos_G$ in \Cref{t:tubingTrig}.

Given a graph $G=(V,E)$, we call connected (not necessarily induced) non-empty subgraphs of $G$ \emph{tubes}, and we denote by $\mathcal{Z}(G)$ the set of tubes of $G$. For a tube $t\in \mathcal{Z}(G)$, we write $V(t)$ and $E(t)$ for its set of vertices and edges, respectively. A tube $t$ is called a \emph{singleton}, if $t$ consists of an isolated vertex, i.e., $V(t)=\{v\}$ for some $v\in V$ and $E(t)=\emptyset$.  
The following definition is of particular importance for us.
\begin{definition}\label{def:tubing}
Let $G=(V,E)$ be a graph. 
    A set $T\subseteq \mathcal{Z}(G)$ is called a \emph{tubing} of $G$ is  if for all $t_1,t_2\in T$ either
    \begin{enumerate}
        \item[(1)] $t_1$ and $t_2$ are \emph{disjoint}, i.e.,  $V(t_1)\cap V(t_2)=\emptyset$, or
        \item[(2)] $t_1$ and $t_2$ are \emph{nested}, i.e., $E(t_1)\subseteq E(t_2)$ or $E(t_2)\subseteq E(t_1)$.
    \end{enumerate} 
    A tubing $T$ is called \emph{maximal} if it is maximal with respect to inclusion.
\end{definition}
 It is not difficult to see that every maximal tubing of $G$ has cardinality $|E|+|V|$. Moreover, since tubes are connected, condition (2) implies that we also have $V(t_1)\subseteq V(t_2)$ or vice versa for all $t_1,t_2\in T$ that satisfy (2). We will also write $t_1\subseteq t_2$ for the condition in (2). We say that two tubes \emph{intersect} each other if none of the two conditions above apply. We denote by $\tau(G)$ the number of maximal tubings of $G$.

\begin{remark}[Different tubing definitions]
    In the literature,  there exist a different notion of a \emph{tubing} than the one we are using, see e.g.,\ \cite{MP2016}. 
    As this notion is connected to graph associahedra, we will call these tubings \emph{GA-tubings}. Before providing the precise definition, we want to recall that a subgraph $H$ of $G=(V,E)$ is called \emph{induced} if $H=(V',\{e\in E~|~e\subseteq V'\})$ for some set $V'\subseteq V$.  
    With this, a \emph{GA-tubing} of $G=(V,E)$ is a collection $T$ of connected induced subgraphs of $G$ that is a tubing in the sense of \Cref{def:tubing} such that for all $t_1, t_2\in T$ with $V(t_1)\cap V(t_2)=\emptyset$, we have 
    \begin{itemize}
        \item[(3)] $E\cap \{v_1v_2 ~|~ v_1\in V(t_1), v_2\in V(t_2)\}=\emptyset$.
    \end{itemize}
    The latter condition means that disjoint tubes have to be non-adjacent.
\end{remark}

Given a graph $G=(V,E)$, its \emph{line graph} is the graph on the vertex set $E$, where two vertices in $L(G)$ are adjacent if the corresponding edges in $G$ are incident. 
Tubings and GA-tubings are related in the following way.

\begin{lemma}
    Let $G$ be a graph and $L(G)$ its line graph. Then there exists a bijection between the tubings of $G$ that contain each singleton and the GA-tubings of $L(G)$.
\end{lemma}

\begin{proof}
To simplify notation, we let $\mathcal{I}(L(G))$ denote the set of vertex-induced connected subgraphs of $L(G)$ and we write $\widetilde{\mathcal{Z}}(G)$ for the set of tubes of $G$ that are not singletons.  
    Since every tube $t\in \widetilde{\mathcal{Z}}(G)$ is determined by its set of edges $E(t)$ and every tube in $s\in\mathcal{I}(L(G))$ by its set of vertices $V(s)$ the map $\phi: \widetilde{\mathcal{Z}}(G)\to \mathcal{I}(L(G))$ that sends $t$ to $L(t)$ is easily seen to be a bijection. 
    Let  $\Phi$ be its natural extension to tubings of $G$, defined by $\Phi(T)=\{\phi(t)~|~ t \in T\}$ for $T\in\widetilde{\mathcal{Z}}(G)$. Since $\phi$ is bijective, it follows that $\Phi$ is injective, and, in particular bijective on its image. 
    
    To show the claim, we hence need to show that the image of $\Phi$ equals the set of GA-tubings of $L(G)$, i.e., we need to prove that  $\Phi(T)$ is a GA-tubing of $L(G)$, if  $T\in \widetilde{\mathcal{Z}}(G)$.

    Let $T\in \widetilde{\mathcal{Z}}(G)$ and let $t_1,t_2\in T$.  If $E(t_1)\subset E(t_2)$, then $V(\phi(t_1))\subseteq V(\phi(t_2))$, by definition of $\phi$, and, since $\phi(t_1)$ and $\phi(t_2)$ are are induced subgraphs, it follows that they are nested as well. 
    If $t_1,t_2\in T$ are disjoint, then $E(t_1)\cap E(t_2)=\emptyset$. This implies that $\phi(t_1),\phi(t_2)$ are disjoint. If $\phi(t_1)$ and $\phi(t_2)$ were adjacent, there would exist an edge $f=v_1v_2\in L(G)$ such that $v_1\in \phi(t_1)$ and $v_2\in \phi(t_2)$. This means, however, that $V(t_1)\cap V(t_2)\neq \emptyset$, which is a contradiction and it follows that $\phi(t_1)$ and $\phi(t_2)$ are non-adjacent. 
    The claim follows. 
\end{proof}

We will now state a few properties of tubings that will be useful later.

\begin{lemma}\label{l:properties_max_tubing}
    Let $G=(V,E)$ be a graph, $T\subseteq\mathcal{Z}(G)$ be a maximal tubing and $t\in T$.
    \begin{enumerate}
     \item[(1)] If $t\neq G$, then there exists a unique successor, i.e., a minimal supertube $\hat{t}\in T$ with $t\subset \hat{t}$.
     \item[(2)] If $t$ is not a singleton, then it has either one or two maximal predecessors, i.e., maximal subtubes $s_1,s_2\subset t$ with $s_1,s_2\in T$. In the first case, $V(t)=V(s_1)$ and $t$ is obtained from $s_1$ by adding a single edge $e$. In the second case $t$ is the union of $s_1$, $s_2$ and an edge $e$ connecting $s_1$ and $s_2$.
     \item[(3)] If $t$ is not a singleton, there exists  a unique edge $e$ such that for all $\tau \in T$ we have $e\in \tau$ if and only if $t\subseteq\tau$. We say that $t$ \emph{introduces} $e$.
    \end{enumerate}
\end{lemma}

\begin{proof}
    To show (1), it suffices to observe that supertubes $\hat{t}_1,\hat{t}_2$ of $t$ have to either contain each other, which makes one non-minimal, or intersect in $t$, which contradicts the fact that $T$ is a tubing. 

    We now show (2). Let $s_1,s_2,\dots, s_\ell$ be maximal predecessors of $t$. Assume by contradiction that $\ell\geq 3$. Due to maximality, $s_1,\ldots,s_\ell$ have to be pairwise disjoint. Since $t$ is connected there has  to exist $2\leq j<\ell$ and an edge $f$ that is incident to a vertex in  $s_1$ and a vertex in $s_j$. The tube $s'$ that is the union of $s_1$ and $s_j$ with the edge $f$ is a predecessor of $t$ and it is different from $t$ since $\ell\geq 3$. 
    This contradicts the maximality of $s_1$ and $s_2$ and it follows that $t$ has one or two maximal predecessors. 
    If it has one maximal predecessor $s_1$, we know by maximality of the predecessor that $t$ is obtained from $s_1$ by adding a single edge (on the same vertex set). If $t$ has two  maximal predecessors $s_1 $ and $s_2$, then, since these have to be disjoint, we know that $t$ is the union of $s_1 $ and $s_2$ together with an edge incident to a vertex in $s_1$ and a vertex in $s_2$.

    Statement (3) directly follows by observing that the edge $e$ from statement (2) is contained in exactly those tubes $\tau\in T$ with $t\subseteq \tau$.  
\end{proof}

\subsection{Cosmological polytopes}\label{sect:basicsCosmo}
Given a graph $G=(V,E)$, we will denote the standard basis vectors in $\R^{|V|+|E|}$ by $\mathbf{x}_{v}, \mathbf{y}_e$ for $v\in V$ and $e\in E$. In \cite{arkani17}, the  \emph{cosmological polytope} $\Cos_G$ of $G$ was defined as follows:

\begin{definition}[Cosmological polytope \cite{arkani17}]
Let $G=(V,E)$ be a graph without isolated vertices. For each edge $e=\{v,w\}\in E$ consider the following three points in $\mathbb{R}^{|V|+| E|}$:
\[
\p_e \coloneqq \x_v+\x_w-\y_e, 
\qquad 
\p_{e,v} \coloneqq \x_v-\x_w+\y_e, 
\qquad 
\p_{e,w}\coloneqq -\x_v+\x_w+\y_e.
\]
The \emph{cosmological polytope} $\Cos_G$ of $G$ is defined as
\[
\Cos_G \coloneqq
\conv\left(
\p_e,\p_{e,v},\p_{e,w}~|~e=\{v,w\}\in E
\right)\subseteq\R^{|V|+|E|}.
\]
\end{definition}
It is easily seen that the set $\{\p_e,\p_{e,v},\p_{e,w}~|~e=\{v,w\} \in E\}$ is the set of vertices of $\Cos_G$, and we will write $V(\Cos_G)$ for this set. 

Moreover, by definition, $\Cos_G$ is contained in the affine hyperplane
\[
H =
\left\{
(x,y) \in \mathbb{R}^{|V|+|E|}
\;\middle|\;
\sum_{v\in V} x_v + \sum_{e\in E} y_e = 1
\right\},
\]
and has dimension $|V|+|E|-1$.
We now recall the facet description of $\Cos_G$, which was provided in \cite[Section 3.1]{arkani17}  (see also \cite[Theorem 2.1]{KM2025}). 

For any tube $t\in\mathcal{Z}(G)$, define the vector 
\begin{equation}\label{eq:normal}
h_t
=
\sum_{v\in V(t)} \x_v
+
\sum_{e\in E\setminus E(t)}
|e\cap V(t)|\, \y_e
\;\in\;
\mathbb{R}^{|V|+|E|}.
\end{equation}
An irredundant facet description of $\Cos_G$ is given by
\[
\Cos_G
=
\left\{
(x,y) \in H
\;\middle|\;
\langle (x,y),h_t\rangle \ge 0
\text{ for all } t\in\mathcal{Z}(G)
\right\}.
\]
In particular, the collection of vectors $\{h_t~|~t\in \mathcal{Z}(G)\}$ is a set of normal vectors of the facets of $\Cos_G$.

In \cite{KM2025}, Kühne and Monin proved the following characterization of the faces of $\Cos_G$, which will be important for us in the next section.

\begin{theorem}\cite[Theorem 3.1]{KM2025}\label{t:CosmoFaces}
Let $G=(V,E)$ be a graph.
A subset $X\subseteq V(\Cos_G)$ defines a face of $\Cos_G$ if and only if $X$ satisfies the following two conditions:
\begin{enumerate}
\item[(1)]
If for a vertex $v\in V$ the set $X$ contains both $\p_e$ and $\p_{e,v}$ for some edge $e=\{v,w\}\in E$, then $X$ contains $\p_{e'}$ and $\p_{e',v}$ for all edges $e'=\{v,w'\}\in E$.

\item[(2)]
If $X$ contains a subset
\[
\{\p_{e_1,v_1}, \p_{e_2,v_2}, \dots, \p_{e_k,v_k}\}
\]
corresponding to a cycle
\[
e_1=\{v_1,v_2\},\;
e_2=\{v_2,v_3\},\;
\dots,\;
e_k=\{v_k,v_1\}
\]
in $G$, then $X$ also contains the subset
\[
\{\p_{e_1,v_2}, \p_{e_2,v_3}, \dots, \p_{e_k,v_1}\}.
\]
\end{enumerate}
\end{theorem}

\subsection{Canonical forms and volumes of dual polytopes}\label{sect:adjoint}
In this section, we will review the representation of the canonical form of a polytope as a volume. All statements relating the canonical form and the volume of polar duals were already sketched in \cite{ArkaniLam17}.
The canonical form is a concept from positive geometry. It is a differential form associated to certain semi-algebraic sets $X$, and it is defined recursively along the boundary of $X$. We will make the definition more explicit after having introduced some further notation.

In the following, let $P\subseteq \R^s$ be a $d$-dimensional polytope on vertex set $V(P)$. 
First recall that the \emph{polar dual} of $P$ is defined as 
\[
P^\circ\coloneqq\{x\in \R^s~|~\langle x,y\rangle\geq -1 \text{ for all } y\in P\}.
\]
If the origin $\mathbf{0}$ is contained in the interior of $P$, then $P^\circ$ is a polytope itself. Moreover, it is well-known that, in this case, vertices of $P$ correspond to facets of $P^\circ$ and vice versa. We refer the reader to \cite{Ziegler} for more background on polar duals and, more generally, polytope theory. 

Let $\Delta$ be any triangulation of $P$ using only vertices of $P$. Going back to Warren \cite{Warren1996}, the \emph{adjoint} of $P$ is the function 
$$  \displaystyle \operatorname{adj}(P)(x) = \displaystyle \sum_{\sigma \in \Delta} \operatorname{vol}(\sigma) \displaystyle \prod_{v \in V(P)\setminus  \sigma} \ell_v,$$ where $x=(x_1,\ldots,x_s)$ and, for a vertex $v\in V(P)$, the form $\ell_v$ is the (up to multiple) unique affine linear form that vanishes one the corresponding facet of the dual polytope $P^\circ$ of $P$. It has been shown that the adjoint is independent of the triangulation $\Delta$ of $P$. 

Polytopes are special cases of so-called \emph{positive geometries} \cite{lam2022}. Their canonical form are related to adjoints in the following way: For any $d$-dimensional projective polytope $P \subseteq \mathbb R^d \subseteq \mathbb P^d$, its \emph{canonical form} $\Omega( P)(x)$ is the following regular differential form
\begin{equation}
\Omega( P)(x)= \frac{\operatorname{adj}(P^\circ)(x)}{ \displaystyle\prod_{F \in \mathcal F} \ell_F(x)}\operatorname{d}(x)\label{denom},\end{equation} where the product in the denominator ranges over the set  $\mathcal{F}$  of facets of $P$ and the form $\ell_F$ is the (up to multiple unique) homogeneous linear form vanishing on the facet $F$ of $P$. The canonical form of $P$ has poles along the facet hyperplanes of $P$, and these poles are simple. For a gentle introduction to the topic, we refer the reader to \cite{lam2022}. 

We would like to  express the forms $\ell_F$ in the denominator of \eqref{denom} in terms of the polar dual of $P$.
In the following, assume that the origin $\mathbf{0}$ is contained in the interior of $P$. Moreover, let $s_1,\ldots,s_k$ be generators of the rays of the normal fan $\mathcal N (P)$ of $P$  such that
$$P = \{x \in \mathbb{R}^d \mid  \langle x, s_i \rangle \geq -1 \text{ for all } 1\leq i\leq  k  \}.$$ 
The vectors $s_i$ are exactly the vertices of the polar dual $P^\circ$ of $P$ and hence, the denominator of \eqref{denom} can be rewritten as $$\prod_{F \in \mathcal F} \ell_F(x) = \prod_{v \in V(P^\circ)} \ell_v(x),$$ where the product on the right-hand side ranges over the set of vertices $V(P^\circ)$ of $P^\circ$.  
By Theorem 5.3 in \cite{gao2024} the canonical form can be computed via the volume of the polar dual as follows 
\[
\operatorname{ vol } (P-x)^{\circ} \operatorname d(x) = \Omega(P)(x)\]
(see also \cite[Section 7.4.1]{ArkaniLam17} and \cite[Theorem 2.5]{Telen}). This shows that the rational form on the left, a priori only defined on the interior of the polytope, uniquely extends to a rational form on $\mathbb P^d$. 
Note that both assignments, $P \mapsto \operatorname{vol}(P-x)^\circ \operatorname{d}(x)$ and $P \mapsto \Omega(P)(x)$ are valuative assignements (\cite{ArkaniLam17},\cite{gao2024}) and therefore both forms agree for any polytope $P$ if we prove the statement for a simplex $ \sigma$, which we will show in the next lemma, suited for our setting.

Note that by the definition of the adjoint, we have
$$\Omega( \sigma)(x)= \frac{\operatorname{adj}(\sigma^\circ)(x)}{ \displaystyle\prod_{F \in \mathcal F} \ell_F(x)}\operatorname{d}(x) = \frac{\displaystyle \operatorname{vol}(\sigma^\circ)}{\displaystyle\prod_{v \in V(\sigma^\circ)} \ell_v(x)}.$$
In the following, we always assume that $x$ is in the interior $\mathrm{int}(P)$ of $P$ so that the functions we consider are well defined. The next statement provides a formula for the volume of the polar dual of a shifted simplex. 

\begin{lemma}\label{lemma:dual}
Let $\sigma$ be a $d$-simplex in $\mathbb R^d \subseteq \mathbb P^d$ with $\mathbf{0}\in \operatorname{int}(\sigma)$. For any $p \in \operatorname{int} (\sigma)$ we have 
    \begin{equation}
    \operatorname{ vol } (\sigma-p)^{\circ}= \frac{\operatorname{vol}(\sigma^\circ)}{\displaystyle \prod_{i=1}^{d+1}{\ell_{v_i}(p)}} \label{volume}.
    \end{equation}
where $v_1,\ldots,v_{d+1}$ are the vertices of the polar dual simplex of $\sigma$ and  $\ell_{v_i}(p) =  {1 + \langle v_i, p \rangle}$.
\end{lemma}
Such a statement was suggested, for example in \cite{gaetz2025} and already proven in more general in \cite{gao2024}. Note that the forms $\ell_{v_i}$ vanish on the corresponding facets of $\sigma$ because the vertices $v_i$ of the dual simplex  are exactly the vectors such that on the corresponding facet we have $\langle v_i, x \rangle = -1$ for all $x$ in that facet.
\begin{proof}
We will show that the vertices of $(\sigma-p)^\circ$ are the vertices of $\sigma^\circ$ scaled by the factors $\ell_{v_i}(p)$. After that, we will show that these special scaling factors $\ell_{v_i}(p)$ factor through the volume, in the sense that $$\operatorname{vol}\left(\Lambda_p(\sigma^\circ)\right) = \left (\prod_{i =1}^{d+1} \ell_{v_i}(p)\right ) \operatorname{vol}(\sigma^\circ),$$ where $\Lambda_p$ is the operator acting on simplices that scales the vertices by the prescribed factors. Note that for an arbitrary choice of scaling factors, this will not hold. \\ 
The shifted simplex $\sigma- p$ has normal fan $\mathcal N (\sigma)$. If $p$ is in the interior of $\sigma$, the shifted simplex is 
\begin{equation}
\sigma -p = \left\{x \in \mathbb{R}^d \mid  \langle x, (\ell_{v_i}(p))^{-1} ~ v_i \rangle \geq -1 \text{ for all } 1\leq i\leq  d+1\}  \right\} \label{shift}
\end{equation}
Therefore, the simplex $(\sigma-p)^\circ$ has vertices $(\ell_{v_i}(p))^{-1} v_i$ (for $1\leq i\leq d+1$) with the notation from above. 
    The volume of the $d$-dimensional simplex $\sigma$ equals the determinant of the matrix $M \in\operatorname{Mat}_{d+1}(\mathbb{R})$ 
    \[
    \begin{pmatrix}
        1 & 1 & \ldots & 1 \\
        v_1 & v_2 & \ldots & v_{d+1} 
    \end{pmatrix}.
    \]
    Similarly, we have 
    \[ \operatorname{vol} (\sigma-p)^\circ =  \frac{\det M'}{\displaystyle \prod_{i= 1}^{d+1}{\ell_{v_i}(p)}} 
    \] where $M'$ denotes the matrix 
     \[
    \begin{pmatrix}
        {\ell_{v_1}(p)} & {\ell_{v_2}(p)} & \ldots & {\ell_{v_{d+1}}(p)}\\
        v_1 & v_2 & \ldots & v_{d+1} 
    \end{pmatrix}
    \]
    by multilinearity of the determinant. 
    Using multilinearity and the definiton of $\ell_{v_i}(p)$ again, we get 
    $$\det(M')  = \det (M) + \det(M'')$$
     where \[  M'' = \begin{pmatrix}
         \langle p, v_1 \rangle & \langle p, v_2 \rangle & \ldots & \langle p, v_{d+1} \rangle \\
        v_1 & v_2 & \ldots & v_{d+1} 
     \end{pmatrix}\]
     
    Since the first row of $M''$ is a linear combination of the last row, it follows that $\det(M'') = 0$ which shows the claim.        
\end{proof}
 
\section{Construction of the dual cosmological polytope}\label{sect:triang}

The cosmological polytope of a graph $G=(V,E) $ lies the affine space $$H = \{(x,y)\in \R^{|V|+|E |}~|~x_1+ \cdots +x_{|V|} + y_1 + \cdots +y_{|E|} = 1\}. $$ The dual polytope in the classical sense is hence not defined. 
To circumvent this issue, we consider $\Cos_G$ as a full-dimensional polytope in the affine hyperplane $H$. In order to describe the vertices of the dual polytope, we make use of the facet description of the cosmological polytope  given by Arkani-Hamed, Benicasa and Postnikov (see \Cref{sect:basicsCosmo}). 

The linear hyperplane normal to $h_t$ (see \eqref{eq:normal}) contains a facet of $\Cos_G.$
Consider the cone $$\operatorname{Cone}(\Cos_G) \coloneqq \left\{ \lambda x \mid \lambda \geq 0, x \in \Cos_G \right\}. $$ We will prove that $\operatorname{Cone}(\Cos_G)^\circ \cap H$ is the dual of the cosmological polytope, where $\operatorname{Cone}(\Cos_G)^\circ$ is the dual cone of $\operatorname{Cone}(\Cos_G)$ if we endow the affine hyperplane $H$ with the right vector space structure.

We justify this construction with the next lemma. In this paper, we use the convention that the \emph{dual cone} of a pointed, full dimensional cone $C \subseteq \mathbb R^{d+1}$ is the cone $$C^\circ \coloneqq \left \{ x \in \mathbb R^{d+1} \mid \langle x, y\rangle \geq 0 \text{ for all } y \in C \right \}.$$

\begin{lemma}\label{lemma:proj}
       Let $C_1$ and $C_2$ be full-dimensional pointed cones in $ \mathbb R^{d+1}$ that are dual to each other. Let $P_1=C_1\cap H$ and $P_2=C_2\cap H$ be the polytopes that arise by intersecting the two cones with a hyperplane $H$ that does not pass through the origin. Then after orthogonal projection via $\pi$ to the linear subspace parallel to $H$ the polytopes $\pi(P_1)$ and $\pi(P_2)$ are dual to each other. 
    \end{lemma}

\begin{proof}

 We first choose some set of coordinates $x_1, \ldots, x_{d+1}$ such that in these coordinates $x_1 = 1$ for any $x\in H$. 
    We consider the affine space $\{x\in \R^{d+1}~|~x_1 =1 \} = H$ as a vector space by projecting to the linear subspace $\{x\in\R^{d+1}~|~x_1=0\}$. Let $x\in H$. Then $$x \in C_1 \iff \langle x, y \rangle \geq 0 \text { for all } y \in C_2 \iff \langle \pi(x), \pi(y) \rangle \geq -1  $$ where the last equivalence holds as $x_1 y_1 = 1$. This shows that $P_1$ and $P_2$ are dual to each other. 
\end{proof}
We will denote the sum of the coordinates of a vector $x$ by $|x|_1$. We are now in the position to be able to provide a vertex description of the polar dual polytope of $\Cos_G$.
\begin{theorem}\label{thm coordinates}
    Let $G=(V,E)$ be a graph with set of tubes $\mathcal{Z}(G)$. Consider $H =  \{(x,y)\in \R^{|V|+|E|}~|~x_1+ \cdots +x_{|V|} + y_1 + \cdots +y_{|E|} = 1\}$ as a vector space equipped with the scalar product induced by the orthogonal projection to the corresponding linear subspace. Then 
    \begin{equation}\label{eq:Vrep}
    \conv \left (\frac{h_t}{|h_t|_1}~|~ t\in\mathcal{\mathcal Z(G)} \right )
    \end{equation}
    is dual to $\Cos_G$ as a full-dimensional polytope in $H$. Moreover, every $z_t\coloneqq\frac{h_t}{|h_t|_1}$ is a vertex of this polytope.
\end{theorem}
We will denote the dual polytope of $\Cos_G$ as in the previous statement by $\DuCos_G$. Before giving the proof, we illustrate this statement via an example.
\begin{example}
    Consider the star graph $K_{1,3}$ with $4$ vertices. It has $11$ connected subgraphs and the list of vertices of $\DuCos_{K_{1,3}}$ are on the right-hand side. Here we assume that the coordinates are ordered as $(v_0, v_1, v_2, v_3, e_1, e_2, e_3)$. 
    
 \vspace{0.5cm}
\begin{minipage}{0.45\textwidth}
    \centering
    \begin{tikzpicture}
        \tikzset{vertex/.style={circle,fill,inner sep=1.3pt}}

        \node[vertex,label=above:$v_1$] (v1) at (90:1.7) {};
        \node[vertex,label=below left:$v_2$] (v2) at (210:1.7) {};
        \node[vertex,label=below right:$v_3$] (v3) at (330:1.7) {};
        \node[vertex,label=above left:$v_0$] (v0) at (0,0) {};

        \draw (v0) -- node[right] {$e_1$} (v1);
        \draw (v0) -- node[above left]  {$e_2$} (v2);
        \draw (v0) -- node[above right] {$e_3$} (v3);
    \end{tikzpicture}

    
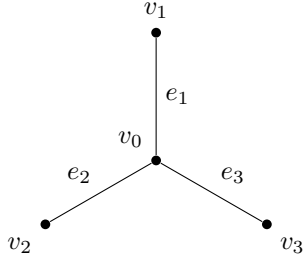
\captionof{figure}{The star graph $K_{1,3}$}
\end{minipage}
\hfill
\begin{minipage}{0.5\textwidth}
 
    \begin{itemize}
        \item $\frac{1}{2}(0,1,0,0,1,0,0)$
        \item $\frac{1}{2}(0,0,1,0,0,1,0)$
        \item $\frac{1}{2}(0,0,0,1,0,0,1)$
        \item $\frac{1}{4}(1,0,0,0,1,1,1)$
        \item $\frac{1}{4}(1,1,0,0,0,1,1)$
        \item $\frac{1}{4}(1,0,1,0,1,0,1)$
        \item $\frac{1}{4}(1,0,0,1,1,1,0)$
        \item $\frac{1}{4}(1,1,1,0,0,0,1)$
        \item $\frac{1}{4}(1,1,0,1,0,1,0)$
        \item $\frac{1}{4}(1,0,1,1,1,0,0)$
        \item $\frac{1}{4}(1,1,1,1,0,0,0)$
    \end{itemize}
\end{minipage}

\end{example}
\vspace{0.5cm}

\begin{proof}[Proof of \Cref{thm coordinates}]
    We let $C_T$ be the cone with rays $h_t$ for $t\in \mathcal{Z}(G)$. We first show that $$C_T =(\operatorname{Cone}(\Cos_G))^\circ$$  Note that $$\operatorname{Cone}(\Cos_G) = \{(x,y) \in \mathbb{R}^{|E|+|V|} \mid \langle( x,y), h_t \rangle \geq 0 \text{ for all } t \in  \mathcal Z(G)  \}$$ by the facet description of $\Cos_G$ (see \Cref{sect:basicsCosmo}).
    
    It is clear that $\langle \lambda h_t, x \rangle \geq 0$ for any tube $t$ and any $x \in \Cos_G$. Therefore, $\langle y, \lambda x \rangle \geq 0$ for any $y \in C_T$ and any  $\lambda \geq 0$ and $C_T \subseteq (\operatorname{Cone}\left(\Cos_G\right))^\circ$ follows. On the other hand, let $x\in ({C_T})^\circ$. Then $\langle x, \lambda h_t\rangle \geq 0$ for any $t \in  \mathcal Z(G)$ and $ \lambda \geq 0$. Hence, $x\in \operatorname{Cone}(\Cos_G)$. 
    Dualizing $({C_T})^\circ \subseteq \operatorname{Cone}(\Cos_G) $ reverses inclusion so $$(\operatorname{Cone}\Cos_G)^\circ \subseteq ({C_T})^{\circ \circ} = C_T$$ 
    where we use that  we work with convex, closed cones. 
    
     We apply now \Cref{lemma:proj} with respect to $H$ and conclude $$\DuCos_G = \left(\operatorname{Cone}(\Cos_G)  \cap H\right)^\circ = C_T \cap H =\conv \left( \frac{h_t}{|h_t|_1}~|~t\in\mathcal{Z}(G) \right).  $$
\end{proof}

\Cref{thm coordinates} allows us to see any triangulation of $\DuCos_G$ without additional vertices as a collection of sets of tubes, where each set has size $|E|+|V|$. A natural choice of such a set system is the set of maximal tubings. Already in the seminal paper \cite{arkani17}, the authors show in some examples that the corresponding collection is a triangulation of $\DuCos_G$. We will prove this statement in full generality. The proof will make use of the following simple lemma:

    \begin{lemma}\label{l:trig2Facets}
        Let $P\subseteq \R^d$ be a $d$-dimensional polytope and let $\Delta$ be a  pure $d$-dimensional simplicial complex on the vertices of $P$. For a simplex $\sigma\in \Delta$, let $\|\sigma\|=\conv(v~|~v\in \sigma)$ be its geometric realization and let $\|\Delta\|=\bigcup_{\sigma\in\Delta}\|\sigma\|$ be the geometric realization of $\Delta$. Assume that for any pair of simplices $\sigma,\tau\in \Delta$, we have that $\dim\sigma=\dim\|\sigma\|$ and $\|\sigma\|\cap\|\tau\|=\|\sigma\cap\tau\|$. If for any $(d-1)$-dimensional face $\sigma\in\Delta$, either 
        \begin{enumerate}
            \item[(1)] $\|\sigma\|\subseteq F$ for some facet $F$ of $P$, or 
            \item[(2)] $\sigma$ is contained in exactly two  facets of $\Delta$,
        \end{enumerate}
        then $\|\Delta\|$ is a triangulation of $P$. 
    \end{lemma}
    We want to remark that the conditions on the dimensions and intersections of the geometric realizations just say that $\|\Delta\|$ is a geometric simplicial complex of dimension $d$. In particular, these conditions are trivially satisfied by any subcomplex of a triangulation of $P$.
    
    \begin{proof} 
    We will show by induction on $d$  that $\|\Delta\| $ is a convex set containing the boundary of $P$. 
    The statement is clear for $d = 1$. Suppose the statement holds for any polytope and any $\Delta$ with the required properties that have dimension at most $d$.
    Let $\dim P = d+1$ and $\Delta$ as in the lemma. For a facet $F$ of $P$ let $\Delta_F\coloneqq \{\sigma \in \Delta~|~\|\sigma\|\subseteq F\}$ be the restriction of $\Delta$ to the vertices of  $F$. We show the following claim.

    {\sf Claim:} $\Delta_F$ is a triangulation of $F$ for any facet $F$ of $P$.
    
    To see this, we show that $\Delta_F$ satisfies the induction hypothesis. For this, we only need to verify conditions (1) and (2) since the other properties are directly inherited from $\Delta$. Let $\sigma\in \Delta_F$ be a $(d-1)$-dimensional face and assume that $\|\sigma\|\not\subseteq G$ for any facet $G$ of $F$. Since $\Delta$ is pure, there exists a facet $\tau^{(1)}\in \Delta$ with $\sigma\subseteq \tau^{(1)}$. In particular, $\tau^{(1)}$ has two $(d)$-dimensional faces $\tau^{(1)}_1$ and $\tau^{(1)}_2$ that contain $\sigma$. Since $\dim\|\tau^{(1)}\|=d+1$, by assumption, at least one of $\|\tau^{(1)}_1\|$ and $\|\tau^{(2)}_2\|$ has to lie in the interior of $P$. Say this is $\|\tau_1^{(1)}\|$. (2) implies that there exists a unique $d$-simplex $\tau^{(2)}$ with $\tau^{(1)}\neq \tau^{(2)}$ and $\tau^{(1)}_1\subseteq \tau^{(2)}$. Moreover, $\tau^{(2)}$ has a unique  $(d)$-dimensional face $\tau_2^{(2)}$ different from $\tau^{(1)}_1$ that contains $\sigma$. If $\|\tau_2^{(2)}\|$ is already contained in the boundary of $P$, then it has to be contained in $F$ already, since otherwise $\|\sigma\|$ would be contained in a facet of $F$. If $\|\tau_2^{(2)}\|$ is contained in the interior of $P$, then we can repeat the argument from before. Since $\|\sigma\|$ is contained in the boundary of $P$, this procedure has to end and after finally many steps, we have found a $(d)$-face $\tau$ with $\|\tau\|\subseteq F$ and $\sigma\subseteq \tau$. Applying the same line of arguments to $\tau_1^{(2)}$, if necessary, shows that $\sigma $ satisfies (2).
    It follows from the claim that $\|\Delta\|$ contains the boundary of $P$.

    It remains to show that $\|\Delta\|$ is convex.
    Assume, by contradiction, that there is a point $p\in P\setminus \|\Delta\|$. Since  $P$ is convex and $\|\Delta\|$ contains the boundary of $P$, there exist points $q_1, q_2\in\|\Delta\|$ and $0 <\lambda<  1$ such that $p=\lambda q_1+(1-\lambda)q_2$.  
    Consider the line segment $L$ between $q_1$ and $q_2$. Let $q'_1$ and $q'_2$ be the closest points contained in $\|\Delta\|\cap L$ that lie between $q_1$ and $p$ respectively $q_2$ and $p$. After possibly wiggling $q'_1$ and $q'_2$, we can assume that there exist $(d-1)$-simplices $\rho_1$ and $\rho_2$ in $\Delta$ such that $q'_i$ lies in the relative interior of $\|\rho_i\|$ for $1\leq i\leq 2$.  Since $\|\rho_i\|$ does not lie in the boundary of $P$, by construction, it must satisfy condition (2). Hence there exist two $d$-simplices containing $\rho_i$. This yields a contradiction since then $q'_i$ was not the closest point to $q$ on $L$. It follows that $p\in \|\Delta\|$ which finishes the proof.
    \end{proof}

    We can finally provide the proof of \Cref{t:tubingTrig}.

\begin{proof}[Proof of \Cref{t:tubingTrig}]
 Let $T_1,\ldots,T_{\tau(G)}$ be the maximal tubings of $G$. For $1\leq i\leq \tau(G)$, let
    \[
\sigma_i = \conv\left( z_t~ |~ t \in T_i \right),
\]
where $z_t=\frac{h_t}{|h_t|_1}$ for a tube $t$. We need to show that $\{\sigma_1,\ldots,\sigma_{\tau(G)}\}$ is a triangulation of
 $\DuCos_G$. 
 To do so, we will apply \Cref{l:trig2Facets}. We first show that $\sigma_i\cap \sigma_j$ is a face of $\sigma_i$ and $\sigma_j$ (if $i\neq j$).
 
To do so, we construct a hyperplane $H_{i,j}$ that separates $\sigma_i$ from $\sigma_j$, i.e., such that $\sigma_i$ and $\sigma_j$ lie on different sides of $H_{i,j}$ and their intersection is contained in $H_{i,j}$. 
For each tube $s \in T_i\setminus T_j$, by maximality of $T_j$, there exists a tube $r \in T_j\setminus T_i$ that intersects $s$. In particular, there exist edges $e =v_1v_2\in E(s)\setminus E(r)$ and $f = v_3v_4\in E(r)\setminus E(s)$ with $v_2,v_3\in s\cap r$.  
    We claim that the linear functional
    $$
        a_{sr}(x,y)\coloneqq x_{v_1}+ x_{v_2} - y_{e} - x_{v_3} - x_{v_4} + y_{f}
    $$
    takes nonnegative values on the vertices of $\sigma_i$ with strictly positive value on $z_s$ and nonpositive values on the vertices of $\sigma_j$ with strictly negative value on $z_r$. Let $t\in \mathcal{Z}(G)$. 
    By construction of $\DuCos_G$ in \eqref{eq:Vrep}, it is clear that $(z_t)_e =  (z_t)_{v_1}+(z_t)_{v_2}$ if $e\notin t$. Moreover, we have $(z_t)_{v_1}=(z_t)_{v_2}=\frac{1}{|h_t|_1}$ and $(z_t)_e=0$ if $e\in t$. We get the corresponding statements, if $f\in t$. Combining this yields

    \[
    a_{sr}(z_t)=\begin{cases}
        \frac{2}{\vert h_t\vert_1} +\frac{-2}{\vert h_t\vert_1}=0, &\qif e\in t \text{ and }f\in t\\ 
         \frac{2}{\vert h_t\vert_1}+0=\frac{2}{\vert h_t\vert_1}>0, &\qif e\in t \text{ and } f\notin t\\
    0 +\frac{-2}{\vert h_t\vert_1}=\frac{-2}{\vert h_t\vert_1}<0, &\qif e\notin t \text{ and } f\in t\\
    0+0 =0,&\qif e\notin t, \text{ and } f\notin t
    \end{cases}.
    \]
    Since no tube in $T_i$  contains $f$ and no tube in $T_j$ contains $e$, the claim from above follows. 
    We can now define the separating hyperplane $H_{i,j}$ for $\sigma_i$ and $\sigma_j$. 
    Let
    $$
      H_{i,j}\coloneqq\{(x,y)\in \R^{|V|+|E|}~|~   \sum_{\substack{s\in T_i,r\in T_j \\ s\text{ intersects } r}} a_{sr}(x,y)=0\}.
    $$
    The previous arguments for $a_{r,s}$ directly imply that $\sigma_i$ and $\sigma_j$ are contained in different sides of $H_{i,j}$ with 
    \[
    H_{i,j}\cap (\sigma_i\cup\sigma_j)=\sigma_i\cap \sigma_j=\conv(z_t~:~t\in T_i\cap T_j).
    \]
    Thus, $H_{i,j}$ is a separating hyperplane and $\sigma_i\cap \sigma_j$ is a face of $\sigma_i$ and $\sigma_j$ as desired.

    Let $\Delta_G$ be the simplicial complex with facets the vertex sets of the simplices $\sigma_t$. In order to apply \Cref{l:trig2Facets} we need to verify that each face of dimension $|V|+|E|-2$, referred to as \emph{ridge} throughout the rest of this proof, satisfies conditions (1) or (2) from this lemma. 

    Let $R\subseteq \sigma_i$ be the ridge, obtained by removing  vertex $v_t$ corresponding to tube $t$ from $\sigma_i$.
    We will distinguish between three cases.

{\sf Case 1: $t$ is a singleton.} By \Cref{l:properties_max_tubing} (1), there  exists a unique smallest tube in $T_i$ that contains $t$, with edge $e=\{v,w\}$ for some $w\in V$. Moreover, we know that each tube $s\in T_i\setminus \{t\}$ that contains $v$ has to contain $e$, too. The vector $\p_{ew}=\x_v-\x_w+\y_e$ is the normal vector of a facet $F$ of $\DuCos_G$ as it is a vertex of $\Cos_G$. Additionally,  the only tubes that do not lie on this facet are the ones that contain $v$ but not $e$, since in every other case the negative contribution of $w$ cancels out the positive contribution from $v$ or $e$.
    This shows that $\|R\|$ is contained in the facet $F$ and thus satisfies condition (1) of \Cref{l:trig2Facets}.

  {\sf Case 2: $t=G$.}  By \Cref{l:properties_max_tubing} (3)  there exists an edge $e$ that $t$ introduces in $T_i$. In particular, $t$ is the only tube in $T_i$ that contains $e=vw$. Consider the facet $F$ of $\DuCos_G$ with facet normal $\p_{e}=\x_v+\x_w-\y_f$. It is easy to see that the only vertices that do not lie on $F$, correspond to tubes that contain $e$. As is Case 1 it follows that $\|R\|\subseteq F$.

  {\sf Case 3: $t\neq G$ and $t$ is not a singleton.}
We will show that there exists a unique $j\neq i$ such that $R\subseteq \sigma_j$. 
    By \Cref{l:properties_max_tubing}, we know that $t$ has a unique successor $\hat{t}$, one or two predecessors $s_1$ (and possibly $s_2$) and that $t$ introduces a unique edge $e$. Similarly, $\hat{t}$ introduces a  unique edge $f$ and has a second (possibly empty) predecessor $s$. 
    Without loss of generality, we can assume that $f$ is incident to a vertex in $s_1$. 
    Let $t'$ the tube obtained by taking the union of $s_1$ and $s$ and the edge $f$. The tube $t'$ is a compatible tube, as it is contained in $\hat{t}$, it contains $s_1$ and $s$ and all their subtubes and $t'$ is disjoint from $s_2$ and all its subtubes. Hence, $t'$ also yields a facet $\sigma_j$ containing a vertex $w$ corresponding to $t'$. So $T_j= (T_i\cup  \{t'\})\setminus \{t\}$ is a tubing, such that $\sigma_j$ contains $R$. By construction, this tubing is unique.

It finally follows from \Cref{l:trig2Facets} that $\{\sigma_1,\ldots,\sigma_{\tau(G)}\}$ is indeed a triangulation of $\DuCos_G$.
    \end{proof}    

    \begin{remark}[Facets of $\DuCos$]\label{r:faceDual}
We have used in the previous proof that, since $\DuCos_G$ are $\Cos_G$ dual to each other, the facet normals of $\DuCos_G$ are given by the vertices of $\Cos_G$.
Since $\Cos_G$ has 3 vertices per edge of $G$, this, in particular, means that $\DuCos_G$ has $3$ facets for each edge $e=\{v,w\}\in G$. 
Namely,
\begin{itemize}
    \item[(1)] $F_{ew}= \{(x,y)\in \Cos_G~|~ -x_v+x_w+y_e=0\}$, which contains all vertices of $\DuCos_G$  except the ones corresponding to tubes containing  $w$ but not $e$,
    \item[(2)] $F_{ev}= \{(x,y)\in \Cos_G~|~ x_v-x_w+y_e=0\}$, which contains all vertices of $\DuCos_G$ except the ones corresponding containing $v$ but not $e$, 
    \item[(3)] $F_{e}= \{(x,y)\in\Cos_G~|~ x_v+x_w-y_e=0\}$, which contains all vertices of $\DuCos_G$ corresponding to tubes that do not contain $e$.
\end{itemize}
More generally, we know that the face lattice of $\Cos_G$ is isomorphic to the face lattice of $\DuCos_G$ if one reverses the order relations.  
This then also yields a bijection between the faces of $\Cos_G$ that contain a certain face $F$ and the faces of $\DuCos_G$ that are contained in its polar $F^\circ$.
\end{remark}

We recall that the \emph{vertex figure} of a $d$-dimensional polytope $P$ at a vertex $v$ is the $(d-1)$-dimensional polytope obtained by intersecting $P$ with a hyperplane that separates $v$ from all the other vertices of $P$. The combinatorial type of this polytope is independent of the chosen hyperplane. More precisely, $k$-faces of the vertex figure correspond to $(k+1)$-faces of $P$ containing $v$ and this extends to an order-preserving bijection between the face lattice of the vertex figure and the corresponding interval in the face lattice of $P$.

\begin{theorem}\label{t:subgraphs_cosmological}
    Let $G=(V,E)$ be a graph and let $e=\{v,w\}\notin E$. Let $G'=(V,E')=(V\cup e,E\cup\{e\}$. Then $\DuCos_G$ is combinatorially equivalent to the facet $F_{e}$ of $\DuCos_{G'}$. Equivalently,  $\Cos_G$ is combinatorially equivalent to the vertex figure of $\p_e$ in $\Cos_{G'}$.
\end{theorem}
We want to emphasize that the edge $e$ that is added to $G$ in the above statement might have a vertex (or even both vertices) not belonging to $V$.

\begin{proof}
Using the well-known duality between the face lattice of $\Cos_G$ and $\DuCos_G$ (see \Cref{r:faceDual}) and the interpretation of the faces of the vertex figure, it is enough to prove that there is an order-preserving bijection between the faces of $\Cos_G$ and the faces of $\Cos_{G'}$ which contain the vertex $\p_e$.
    
   Since $\Cos_G$ is not changed when we add isolated vertices to $G$, we can assume that $v,w\in V$.
     Let $F$ be a facet of $\Cos_G$. We will now define a facet of $\Cos_{G'}$ containing $e$. To do so, we consider different cases.

     {\sf Case 1:} If  $\{\p_{f},\p_{f,v}\}\not\subseteq F$ for any  edge $f=vx\in E$ and $\{\p_{g},\p_{g,w}\}\not \subsetneq F$ for any $g=wy\in E$, we set $F'=\conv(F\cup\{\p_e\})$. 

    {\sf Case 2:} If there exists $f=vx\in E$ such that $\{\p_{f},\p_{f,v}\}\subseteq F$ and $\{\p_{g},\p_{g,w}\}\not \subsetneq F$ for any $g=wy\in E$, then we set $F'=\conv( F\cup \{\p_e, \p_{e,v}\})$. Note that in this case, it follows from \Cref{t:CosmoFaces} (1), that $\{\p_{h},\p_{h,v}\}\subseteq F$ for all $h=vz\in E$.

    {\sf Case 3:} If $\{\p_{f},\p_{f,v}\}\not\subseteq F$ for any  edge $f=vx\in E$ and there exists $g=ey\in E$ with $\{\p_{g},\p_{g,w}\} \subseteq F$, then we set $F'=\conv( F\cup \{\p_e, \p_{e,w}\})$. Note that in this case, it follows from \Cref{t:CosmoFaces} (1), that $\{\p_{k},\p_{k,v}\}\subseteq F$ for all $k=wz\in E$. 

    {\sf Case 4: }If there exists $f=vx\in E$ such that $\{\p_{f},\p_{f,v}\}\subseteq F$ and there exists $g=ey\in E$ with $\{\p_{g},\p_{g,w}\} \subseteq F$, then we set $F'=\conv( F\cup \{\p_e,\p_{e,v}, \p_{e,w}\})$. Again \Cref{t:CosmoFaces} (1) implies that we have $\{\p_{h},\p_{h,v}\}\subseteq F$ for all $h=vz\in E$ and $\{\p_{k},\p_{k,v}\}\subseteq F$ for all $k=wz\in E$ in this case.

    To show that $F'$ is a face of $\Cos_{G'}$ in all four cases, we need to verify that conditions (1) and (2) from \Cref{t:CosmoFaces} are satisfied. This is clear as long as $v,w$ and $e$ are not involved. Condition (1) is also clear around  vertex $v$ or $w$, by definition of $F'$ 

    For condition (2),  let $C= e_1=vu_1,\dots, e_{k-1}=u_kw,e=vw$ be a cycle in $G'$ that contains $e$ and assume that 
    $\p_{e_1,v},\dots,\p_{e_{k+1},u_k},\p_{e,w}\in F'$. Then, by construction, $F'$ also contains $\p_e$ and therefore by \Cref{t:CosmoFaces} (1) we know that it also contains $\p_{e_{k+1}}$ and $\p_{e_{k+1},w}$. 
    As $F'$ contains $\p_{e_{k+1}}$ and $\p_{e_{k+1},u_k}$ it also contains $\p_{e_k}$ and $\p_{e_k,u_{k-1}}$, which then inductively also yields that it must contain $\p_{e_i,u_{i-1}}$ for $1\leq i\leq k$ and lastly also $\p_{e,v}$, which proves that \cref{t:CosmoFaces} (2) is also satisfied. 

We need to compute the dimension of $F'$. Since, we always have that $F\subseteq \{(x,y)\in \R^{|V|+|E|}~|~y_e=1\}$ but $\p_e\notin \{(x,y)\in \R^{|V|+|E|}~|~y_e=1\}$, it follows that $\dim F+1\leq \dim F'\leq \dim F+2$. Since in the first case, we only add one vertex, it is clear that $\dim F'=\dim F+1$.  In the second case, $F'$ contains the vertices $\p_e$ and $\p_{e, v}$, and then also needs to contain $\p_{f},\p_{f,v}$ for some $f=vx\in E$. Since we have the affine dependency $\p_{e,v}=\p_{f,v}+\p_{f}-\p_{e}$, we see that $\dim(F')=\dim(F)+1$. The same argument, also applied to $w$, yields that $\dim(F')=\dim(F)+1$ in Case 3 and 4.

It is clear from the definition that this map from the set of faces of $\Cos_G$ to the set of faces of $\Cos_{G'}$ that contain $\p_e$ is order-preserving and moreover, that it can be inversed. To see that the inverse map is well-defined it is enough to observe that  both conditions in \Cref{t:CosmoFaces} are closed under taking subgraphs. The claim follows.
\end{proof}

Applying \Cref{t:subgraphs_cosmological} inductively yields the following corollary. 

\begin{corollary}
    Let $G$ be a graph and $H\subseteq G$ be a subgraph. Then there exists a face of $\DuCos_G$ that is combinatorially equivalent to $\DuCos_H$.
\end{corollary}

\begin{remark}[Connection to graph Associahedra]    
From the proof of \Cref{t:tubingTrig} one can also extract that two simplices in the triangulation of $\DuCos_G$ share a facet if and only if their tubings differ by exactly two tubes, i.e., one can construct one tubing from the other by simply removing one tube and replacing it by a different (compatible) one. 
On the other hand, the graph associahedron of $G$ is the simple polytope, whose vertices correspond to maximal GA-tubings of $G$, where two vertices are connected by an edge if and only their GA-tubings differ by exactly two tubes. 
This means that the $1$-skeleton of the graph associahedron of $G$ is a subgraph of the facet-ridge graph of the triangulation of $\DuCos_G$ from \Cref{t:tubingTrig}. 
\end{remark}

\section{Computing the canonical form}\label{section_comp}
The goal of this section is to apply our results from the previous sections to compute the canonical form of a cosmological polytope. We will provide two different representation: Whereas the first one is well-known at least in the cosmology community, we have not seen the second one before.

Given a graph $G=(V,E)$, we write $\mathcal{T}(G)$ for its set of maximal tubings. Moreover, for a tube $t\in\mathcal{Z}(G)$ and $x\in \mathbb{P}^{|E|+|V|-1}$,  we set $x_t\coloneqq  \langle x, h_t \rangle$. 

With these notations, we can provide the first representation of the canonical form, which also has been computed in \cite{arkani17} by different means:

\begin{theorem}  \label{canonicalform}
    Let $G=(V,E)$ be a graph. The canonical form of the cosmological polytope $\Cos_G$ is, up to a scalar, given by 
\begin{equation}
\Omega(\Cos_G)(x) = \sum_{T \in \mathcal T(G)} \frac{\operatorname{vol}(\sigma_T)}{\displaystyle \prod_{t \in T} x_t}\operatorname d(x) \label{can},
\end{equation}
where $\sigma_T=\conv(z_t~|~t\in T)$ is the simplex whose vertices correspond to the tubes in $T$. 
\end{theorem}

\begin{proof}
By \Cref{canonical}, the canonical form of $\Cos_G$ is given by
$$ \Omega(\Cos_G)(x) = \operatorname{vol}(\Cos_G-x)^\circ\operatorname d(x).$$ 
We triangulate $\DuCos_G$ via the triangulation given in \Cref{t:tubingTrig} and use valuativity to write this as 
$$ \displaystyle \Omega(\Cos_G)(x) =\sum_{T \in \mathcal T} \operatorname{vol}(\sigma_T(x))\operatorname d(x),$$
where the simplices $\sigma_T(x)$ are just the simplices of the induced triangulation on the polar dual of the shifted cosmological polytope. We saw in  \Cref{lemma:dual} that 
$$\operatorname{vol}(\sigma_T(x))= \frac{\operatorname{vol}(\sigma_T)}{\displaystyle\prod_{t  \in T}\ell_t(x)},$$ where $\ell_t(x)$ is the form $\ell_v(x)$ in the mentioned lemma for the vertex $v$ of the dual corresponding to the tube $t$. 
Summing this up we obtain
$$\Omega(\Cos_G)(x) = \sum_{T \in \mathcal T}\frac{\operatorname{vol}(\sigma_T)}{\displaystyle\prod_{t \in T} \ell_t(x)}\operatorname d(x).$$

This is exactly the expression \eqref{can} in the claim but lifted to the projective setting. The forms $x_t$ are projective linear forms vanishing on the facet corresponding to the tube $t$ and the form $\ell_t(x)$ is the affine linear form vanishing on the corresponding facet. Such a form is unique in both the real and the projective setting and therefore the claim follows. \end{proof} 

\remark[Including the adjoint]
\Cref{canonicalform} can also be proven using the adjoint. 
We saw earlier that 
\begin{equation}\label{canonCosmo}
    \Omega( \Cos_G)(x)= \frac{\operatorname{adj}(\Cos_G^\circ)(x)}{ \displaystyle\prod_{F \in \mathcal F} \ell_F(x)}\operatorname{d}(x),
\end{equation}
where the product ranges over the set of facets $\mathcal F$ of $\Cos_G$ and the forms $\ell_F$ are the corresponding linear forms vanishing on the affine hulls of the facets. Using the definition of the adjoint (see \Cref{sect:adjoint})
and the triangulation of $\DuCos_G$ from \Cref{t:tubingTrig} we conclude
\begin{equation}\label{adjointCosmo}
\operatorname{adj}(\DuCos_G) = \displaystyle \sum_{T \in \mathcal T(G)} \operatorname{vol}(\sigma_T) \prod_{t \in\mathcal{Z}(G)\setminus  T} \ell_t(x).\end{equation}
Combining \eqref{canonCosmo} and \eqref{adjointCosmo} also yields the desired formula.

\subsection{Computing the canonical form via the boundary complex}
We will now provide a second way to compute the canonical form of a cosmological polytope. The idea is to modify the triangulation from \Cref{t:tubingTrig} slightly and to compute the volume of the polar dual of the shifted cosmological polytope via this new triangulation.
More precisely, we will allow the simplices in this new triangulation to use an additional vertex that is not a vertex of $\DuCos$. 

This idea is formulated in the following well-known lemma which is straightforward to prove.

\begin{lemma}\label{lemma:boundary}
    Let $P\subseteq \R^d$ be a polytope and let $p$ be a point in the interior or $P$. If  $\mathcal{S}$ is a  triangulation of the boundary of $P$, then
    \[
    \mathcal{S}\ast p\coloneqq\left(\conv(\sigma\cup \{p\})~|~\sigma\in \mathcal{S}\right)
    \]
    is  a triangulation of $P$.
    \end{lemma}

We start by describing the induced triangulation of the boundary of $\DuCos_G$ when restricting the triangulation of $\DuCos_G$ from \Cref{t:tubingTrig} to the boundary of $\DuCos_G$. We call a tubing $T$ of a graph $G=(V,E)$ \emph{almost maximal} if $|T|=|V|+|E|-1$. Moreover, an almost maximal tubing is called \emph{uniquely completable} if it was obtained from a maximal tubing by removing either a singleton or the graph $G$. Note that, by definition, uniquely completable, almost maximal tubings are exactly the almost maximal tubings that are contained in a unique maximal tubing.

\begin{lemma}\label{l:newTriang}
    Let $G=(V,E)$ be a graph. Let $\widetilde{\mathcal{T}}(G)$ be the set of uniquely completable, almost maximal tubings of $G$. Then 
    \[
     \{\sigma_T~|~T\in \widetilde{\mathcal{T}}(G)\},
    \]
    where $\sigma_T\coloneqq\conv(z_t~|~t\in T)$ is a triangulation of the boundary of $\DuCos_G$.
\end{lemma}
\begin{proof}
Let $\mathcal{S}(G)$ be the triangulation of $\DuCos_G$ from \Cref{t:tubingTrig}. By definition, every simplex in $\mathcal{S}(G)$ contains the vertices of $\DuCos_G$ corresponding to the singletons in $G$ and the graph $G$ itself.
 The ridges of the simplices in $\mathcal{S}(G)$ correspond to almost maximal tubings. Moreover, by the discussion preceding this lemma, ridges lying  on the boundary of $\DuCos_G$ correspond to uniquely completable, almost maximal tubings of $G$. The claim follows.
\end{proof}
 
\begin{example}
    Let $G$ be the path of length two. 
    
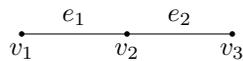
\begin{figure}[H]
\begin{center}
\begin{tikzpicture}[scale=0.7]

\fill (0,0) circle (1.5pt) node[below] {$v_1$};
\fill (2,0) circle (1.5pt) node[below] {$v_2$};
\fill (4,0) circle (1.5pt) node[below] {$v_3$};

\draw (0,0) -- (2,0) node[midway, above] {$e_1$};
\draw (2,0) -- (4,0) node[midway, above] {$e_2$};

\end{tikzpicture}
\end{center}
\end{figure}

There are two maximal tubings of the graph given by choosing the singletons, the full graph and additionally either the edge $e_1$ or the edge $e_2$. Let $\sigma_1$ and $\sigma_2$ be  the corresponding maximal simplices in the triangulation of $\DuCos_G$. The $5$ facets of $\sigma_1$ are given by the following $3$-dimensional simplices

\begin{enumerate}
    \item[(1)] $\operatorname{conv}\left (\{z_{v_1}, z_{v_2}, z_{v_3}, z_{G}\}\right)$
    \item[(2)] $\operatorname{conv}\left (\{z_{v_1}, z_{v_2}, z_{v_3}, z_{e_1}\}\right)$
    \item[(3)] $\operatorname{conv}\left (\{z_{v_1}, z_{v_2}, z_{e_1}, z_{G}\}\right)$
    \item[(4)] $\operatorname{conv}\left (\{z_{v_1}, z_{e_1}, z_{v_3}, z_{G}\}\right)$
    \item[(5)] $\operatorname{conv}\left (\{z_{e_1}, z_{v_2}, z_{v_3}, z_{G}\}\right)$
\end{enumerate}
The first simplex is in the intersection with $\sigma_2$. The other simplices are on the boundary of $\DuCos_G$. For the other simplex $\sigma_2$ we get the same collection of simplices again, but where we replace $z_{e_1}$ by $z_{e_2}$. Therefore, the boundary of $\DuCos_G$ is triangulated by eight simplices. Coning over an interior point of $\DuCos_G$ (see \Cref{lemma:boundary}) produces a triangulation of  $\DuCos_G$ with eight simplices. 
\end{example}

Using  \Cref{l:newTriang} to compute the canonical form of the the cosmological polytope $\Cos_G$ of a graph $G$, we obtain the following expression for $\Omega(\Cos_G)$:

\begin{theorem}\label{t:canonicalBoundary}
Let $G=(V,E)$ be a graph. 
Then,
    $$\Omega (\Cos_G)(x) =\displaystyle \sum_{ T \in \widetilde{\mathcal T}(G)}\frac{\operatorname{vol}{\sigma_{T}}}{\displaystyle \prod_{t \in T}x_t }\operatorname{d}(x),
    $$
    where $\sigma_T=\conv( \{\frac{1}{|V|+|E|}\cdot\mathbf{1}\}\cup\{z_t~|~t\in T\})$ for $T\in \widetilde{\mathcal{T}}(G)$.
\end{theorem}
\begin{proof}
Let $\mathbf{1}\in\R^{|V|+|E|}$ be the all ones vector. We note that $\frac{1}{|V|+|E|}\cdot \mathbf{1}$ is in an interior point of $\DuCos_G$. Combining \Cref{l:newTriang} and \Cref{lemma:boundary}, we get that the dual cosmological polytope $\DuCos_G$ is triangulated by the set of simplices  $\{\sigma_{T}~|~T\in \widetilde{\mathcal{T}}(G)\}$ and this triangulation induces a triangulation of the dual of the shifted cosmological polytope.  We use once more \Cref{t:tubingTrig} to compute the canonical form of $\Cos_G$. To do so, we need to compute the volume of the simplices $\sigma_T$. For this, we look at their image under orthogonal projection to the linear subspace parallel to $H = \{(x,y)\in \mathbf{R}^{|V|+|E|}~|~x_1+x_2+ \cdots +x_{|V|}+ y_1+ \cdots + y_{|E|}=1 \}$. Note that the image of $\sigma_T$ is just given by the translation $\sigma_T-\frac{1}{|V|+|E|}\cdot \mathbf{1}$. Denoting by $\sigma_T(x)$ the corresponding simplex in the induced triangulation of $(\Cos_G-x)^\circ$ and applying \Cref{lemma:dual}, we get that
\[
\operatorname{vol}(\sigma_T(x))=\frac{\operatorname{vol}(\sigma_T)}{\prod_{t\in T}\ell_t(x)},
\]
where we use that the origin $\mathbf{0}$ is a vertex of $\sigma_T-\frac{1}{|V|+|E|}\cdot \mathbf{1}$ and $\ell_{\mathbf{0}}(x)=1$.
Summing over all tubings in $\widetilde{\mathcal{T}}(G)$ yields
\[
\Omega (\Cos_G)(x) =\displaystyle \sum_{ T \in \widetilde{\mathcal T}(G)}\frac{\operatorname{vol}{\sigma_{T}}}{\displaystyle \prod_{t \in T}\ell_t(x) }\operatorname{d}(x)
\]
The claim follows by the same argument as in the proof of \Cref{canonical}.
\end{proof}

\begin{remark}
    The advantage of the representation of the canonical form in \Cref{t:canonicalBoundary} compared to the expression in \Cref{canonicalform} is that the degree of the rational forms occurring as summands is one lower, the disadvantage is we have $|V|+1$ many summands. It is not clear to us what is more desirable from a physical perspective.
\end{remark}
\subsubsection*{Acknowledgement}
We thank Hadleigh Frost and Felix Lotter for helpful discussions concerning the results in \Cref{section_comp}. 

This work has been supported by the German Research Foundation (DFG) via SPP 2458 \textit{Combinatorial Synergies} and via project number 547295909.

\bibliography{bibligraphy.bib}
\bibliographystyle{plain}
\end{document}